\documentclass[12pt]{amsart}
\pdfoutput=1
\usepackage{microtype}
\usepackage{etex} 
\usepackage[active]{srcltx}
\usepackage{calc,amssymb,amsthm,amsmath,amscd, eucal,ulem}
\usepackage{alltt}
\RequirePackage[dvipsnames,usenames]{color}

\input{mabliautoref.sty}
\input{xy}
\xyoption{all}
\input{kmacros3.sty}
\usepackage{tikz}
\usepackage{amsfonts, mathrsfs}
\usepackage{mathtools}

\usepackage{calligra}
\usepackage{stmaryrd} 
\usepackage{dcpic, pictexwd}
\usepackage[left=1in,top=1in,right=1in,bottom=1in]{geometry}
\usepackage{bm}
\usepackage{verbatim}

\usepackage[cal=boondox, calscaled=1.05]{mathalfa}

\numberwithin{equation}{theorem}

\renewcommand{\m}{\mathfrak{m}}

\DeclareMathOperator{\gr}{gr}

\def\m{\mathfrak{m}}
\def\b{\mathfrak{b}}
\def\a{\mathfrak{a}}

\usepackage{cleveref}

\usepackage{setspace}
\usepackage{hyperref}

\usepackage{url}
\let\oldFootnote\footnote
\newcommand\nextToken\relax

\renewcommand\footnote[1]{%
    \oldFootnote{#1}\futurelet\nextToken\isFootnote}

\newcommand\isFootnote{%
    \ifx\footnote\nextToken\mathrmsuperscript{,}\fi}

\usepackage{enumitem}
\usepackage{graphicx}
\usepackage[all,cmtip]{xy}

\usepackage{verbatim}

\theoremstyle{theorem}

\newtheorem{theoremx}{Theorem}
\newtheorem{corollaryx}[theoremx]{Corollary}

\usepackage{marvosym}

\begin{document}
\title[{$F$-thresholds and  $F$-rationality}]{On the behavior of $F$-thresholds with respect to the fibers of blow-ups and  $F$-rationality}

\author{Ghazaleh FakhriVaighan}
\address{Centro de Investigaci\'on en Matem\'aticas, A.C., Callej\'on Jalisco s/n, 36024 Col. Valenciana, Guanajuato, Gto, M\'exico}
\email{\href{mailto:ghazaleh.fakhrivaighan@cimat.mx}{ghazaleh.fakhrivaighan@cimat.mx}}


\thanks{The author acknowledges the support of SECIHTI Fellowship \#1321301 and Grant CBF 2023-2024-224. }



\begin{abstract}

We establish a general inequality comparing the $F$-thresholds of a local ring and its associated graded ring. As an application, we deduce that the $F$-rationality of the graded ring descends to the local ring.
\end{abstract}
\maketitle

\section{Introduction}

The study of singularities lies at the crossroads of commutative algebra and algebraic geometry. In positive characteristic, singularities is captured by a family of Frobenius-based concepts referred to as $F$-singularities. Among these, $F$-pure \cite{HR, Fed} and $F$-regular singularities \cite{HHS, HHT, HHP, FW} play a central role, reflecting vanishing and splitting phenomena. To quantify their behavior, the principal numerical invariants are the $F$-thresholds \cite{MTW, HMTW}, which capture the severity of a singular point through the action of Frobenius powers.

The concept of $F$-thresholds was first introduced by Mustaţă, Takagi, and Watanabe \cite{MTW}, and further developed in their joint work with Huneke \cite{HMTW}. The $F$-threshold of $\mathfrak{a}$ with respect to $J$ denoted by $c^{J}(\mathfrak{a})$ is defined in terms of a limit which was shown to exist in full generality later by De Stefani, Núñez-Betancourt, and Pérez \cite{DSNBP}. These invariants capture subtle asymptotic data about the Frobenius powers of ideals. They are intertwined with tight closure, Hilbert–Samuel multiplicity, and integral closure, and thus provide a flexible bridge between Frobenius methods and numerical invariants of singularities.

Given its importance, a natural and pressing question is how $F$-thresholds behaves under standard constructions. One such construction is the blowup of a local ring at its maximal ideal. If $(R,\mathfrak{m})$ is a Noetherian local ring, the blowup of $\mathrm{Spec}(R)$ at $\mathfrak{m}$ is
$$
\Proj\Bigg(\bigoplus_{i\geq 0}\mathfrak{m}^i\Bigg),
$$
and its special fiber is precisely
$$
\Proj \big(\mathrm{gr}_\mathfrak{m}(R)\big),$$ where 
$$
\mathrm{gr}_\mathfrak{m}(R) = \bigoplus_{i\geq 0}\mathfrak{m}^i / \mathfrak{m}^{i+1}
$$
is the associated graded ring of $R$ with respect to $\mathfrak{m}$. This graded ring encodes the asymptotic structure of the powers of $\mathfrak{m}$, and thus reflects deep information about the singularity of $(R,\mathfrak{m})$.


Given an $\mathfrak m$-primary ideal in a local ring, it is natural to ask how the corresponding $F$-threshold compares with the $F$-threshold of its degeneration to the associated graded ring? 

To make this precise, let $\mathfrak b\subseteq R$ be an $\mathfrak m$-primary ideal, and let
\[
\mathrm{gr}_\mathfrak m(\mathfrak b)
  \;=\;
  \bigoplus_{i\ge 0}
  \frac{\mathfrak b\cap \mathfrak m^i}{\mathfrak b\cap \mathfrak m^{i+1}}
\]
denote its initial ideal in the associated graded ring $\mathrm{gr}_\mathfrak m(R)$. This question is addressed by the following theorem.

\begin{theoremx}[{\autoref{20}}]\label{slm}
 Let $(R, \mathfrak{m}, k)$ be a local $F$-finite ring of prime characteristic $p>0$, and $\mathfrak{b}\subseteq R$ be an $\mathfrak{m}$-primary ideal. Then,
$$c^{\mathfrak{b}}(\mathfrak{m}) \leq c^{I}(\mathfrak{n}),$$ where $\mathfrak{n}$ is the maximal homogeneous ideal of the associated graded ring  $\mathrm{gr}_{\mathfrak{m}}(R)$ and $I=\mathrm{gr}_\m(\b)$.

\end{theoremx}

This inequality holds without assuming Cohen--Macaulayness or the Gorenstein property, and it allows one to transfer Frobenius numerical data from the graded fiber back to the local ring.

As a consequence of our comparison theorem for $F$-thresholds, we obtain a corollary concerning $F$-rationality (see \autoref{21}). Although the statement might have been known to experts (see \autoref{otherproof}), we  could not find it in the literature until we were made aware of Ajit and Simper's work on test modules of extended Rees algebras \cite[Theorem 6.1(2) and Remark 6.2]{AS}. Our argument provides a novel proof that proceeds through the behavior of $F$-thresholds. Although our proof is longer, for the sake of completeness, we provide proofs of statements that are known for graded rings, but that we could not find in the literature (see \autoref{2}, \autoref{3},  \autoref{PropGenParameterIdeal}), and \autoref{22}).

A result of Huneke, Mustaţă, Takagi, and Watanabe \cite{HMTW} characterizes $F$-rationality via $F$-thresholds: a local ring $(R,\mathfrak{m})$ is $F$-rational if and only if the $F$-threshold of its maximal ideal is strictly less than $\dim(R)$. This numerical characterization transforms the problem of detecting $F$-rationality into a problem about Frobenius asymptotics, by passing the need to test the tight closure of parameter ideals directly. Combining \autoref{slm} with the $F$-threshold characterization of $F$-rationality, we obtain the following consequence of our main theorem.

\begin{corollaryx}[{\autoref{21}}]
Let $ (R, \mathfrak{m}, k)$ be a $d $-dimensional $F$-finite Noetherian local domain of characteristic $p > 0$. If the associated graded ring  $\mathrm{gr}_{\mathfrak{m}}(R)$ is $F$-rational, then $R$ is $F$-rational.
\end{corollaryx}

\subsection*{Acknowledgments} 
The author thanks Javier~Carvajal-Rojas and Luis~Núñez-Betancourt for helpful conversations. The author also thanks Alessandro~De Stefani for helpful comments regarding $F$-rationality.

\section{Notations and Preliminaries}\label{notation}

In this section, we recall the main notions and terminology used throughout the paper.  Let $R$ be a Noetherian ring of characteristic $p > 0$. For any $e > 0$, the $e$-th iteration of the Frobenius map, $F^e : R \to R$, is defined by $F^e(f) = f^{p^e}$ for $f \in R$. This map gives $R$ a new $R$-module structure via restriction of scalars, which we denote by $F^e_*(R)$ and its elements by $F_*^e(f)$. The operations in this module satisfy:
$F^e_*(f_1) + F^e_*(f_2) = F^e_*(f_1 + f_2)$, and
$f_1F^e_*(f_2) = F^e_*(f_1^{p^e} f_2)$ for all $f_1, f_2 \in R$.

We begin by recalling some standard notions from the theory of $F$-singularities, which are used throughout our later discussions.

\begin{definition}

Let $R$ be a Noetherian ring of prime characteristic $p>0$. 
\begin{enumerate}
\item We say that $R$ is $F$-finite if $F^e_*(R)$ is a finitely generated $R$-module for some (equivalently, all) $e > 0$.
\item We say that $R$ is $F$-pure if the Frobenius map $F^e: R \to F_*^eR$ is pure for some (equivalently, all) $e > 0$.
\item We say that $R$ is strongly $F$-regular if it is reduced, and for every $c \in R\setminus \bigcup_{\mathfrak{p}\in \mathrm{Min}(R)} \mathfrak{p}$ there exists $e \in \mathbb{Z}_{>0}$ and an $R$-linear map $\phi: F^e_*(R) \to R$ such that $F^e_*(c)\mapsto 1$.
\item The tight closure $I^*$ of $I$ is defined to be the
ideal of $R$ consisting of all elements $x \in R$ for which there exists $c \in R$ such that $cx^q \in I^{[q]}$ for all large $q = p^e$.
\item We say that $R$ is $F$-rational if $J^*= J$ for every ideal $J \subseteq R$ generated by parameters. If $R$ is an excellent equidimensional local ring, then $R$ is $F$-rational if
and only if $I = I^*$ for some ideal $I$ generated by a full system of parameters for $R$.
\end{enumerate}

\end{definition}

We next consider two numerical invariants used to study singularities in characteristic $p$: the $F$-pure threshold and the $F$-threshold. These invariants provide a framework for analyzing how singularities interact with the Frobenius endomorphism.

\begin{definition}[\cite{AE}] Let $(R, \mathfrak{m}, k)$ be either a local ring or a standard graded $k$-algebra. For each $e \in \mathbb{N}$, we set
\[
I_e = \left\{ f \in R \mid \psi(F^e_* f) \in \mathfrak{m} \,\text{ for all $R$-linear maps} \,\,\psi : F^e_* R \to R \right\}.
\]  
In particular, $I_0=\mathfrak{m}$.
\end{definition}

\begin{definition}[\cite{TW}]\label{fpt}
Let $(R, \mathfrak{m}, k)$ be either a local ring or a standard graded $k$-algebra. Assume that $R$ is an $F$-pure ring. For $e \in \mathbb{N},$ we associate to the ideals $I_e$ the following integers:
\[
b_{\mathfrak{a}}(p^e)= \max \left\{ t \in \mathbb{N} \mid \mathfrak{a}^t \nsubseteq I_e \right\}.
\]  
Given a proper ideal $\mathfrak{a} \subseteq R$, homogeneous when $R$ is graded, we define the $F$-pure threshold of $\mathfrak{a}$ in $R$ as  
\[
\mathrm{fpt}(\mathfrak{a})= \lim_{e \rightarrow\infty} \frac{b_{\mathfrak{a}}(p^e)}{p^e}.
\]
When $\mathfrak{a} = \mathfrak{m}$, the $F$-pure threshold $\mathrm{fpt}(\mathfrak{m})$ is often simply denoted by $\mathrm{fpt}(R)$.
\end{definition}

\begin{definition}[\cite{HMTW,DSNBP}]
Let $R$ be a ring of prime characteristic $p$. For ideals $\mathfrak{a}, J \subseteq R$ satisfying $\mathfrak{a} \subseteq \sqrt{J}$, and a non-negative integer $e$, the \emph{$F$-threshold} of $\mathfrak{a}$ with respect to $J$ is defined as
\[
c^{J}(\mathfrak{a})= \lim_{e \rightarrow\infty} \frac{\nu^{J}_{\mathfrak{a}}(p^e)}{p^e},
\] where 
\[
\nu^J_{\mathfrak{a}}(p^e)= \max \left\{ t \in \mathbb{N} \mid \mathfrak{a}^t \not\subseteq {J}^{[p^e]} \right\}.
\]
\end{definition}

The following theorem collects some basic monotonicity properties of $F$-thresholds.

\begin{theorem}[{\cite[Proposition 1.7]{MTW}}]\label{9}
Let $\mathfrak{a}$, $J$ be ideals as above.
\begin{enumerate}
\item If $I \supseteq  J$, then $c^I(\mathfrak{a}) \leq c^J(\mathfrak{a}).$
\item If $\mathfrak{b} \subseteq \mathfrak{a}$, then $c^J(\mathfrak{b}) \leq c^J(\mathfrak{a})$. Moreover, if $\mathfrak{a} \subseteq \bar{\mathfrak{b}}$, where $\bar{\mathfrak{b}}$ denotes the integral closure of $\mathfrak{b}$, then $c^J(\mathfrak{b}) = c^J(\mathfrak{a})$.
\end{enumerate}
\end{theorem}

\begin{definition}
Let $(R,\mathfrak{m},k)$ be a Noetherian local ring. The associated graded ring of $R$ with respect to its maximal ideal $\mathfrak{m}$ is the graded ring
$$\mathrm{gr}_\mathfrak{m}(R)=\bigoplus_{i=0}^{\infty}\mathfrak{m}^i/\mathfrak{m}^{i+1},$$ with maximal ideal $\mathfrak{n}=\bigoplus_{i=1}^{\infty}\mathfrak{m}^i/\mathfrak{m}^{i+1}$. 

For $a\in \mathfrak m^i$, we denote by
$$
[a]= a+\mathfrak m^{i+1}\in \mathfrak m^i/\mathfrak m^{i+1}
$$
the coset of $a$ in the $i$-th graded piece. 

The $\mathfrak m$-adic order of a nonzero element $a\in R$ is defined as
$$
\ord_\mathfrak m(a)=\max\{\,r\ge 0 \mid a\in \mathfrak m^r\,\}.
$$
Equivalently, $\ord_\mathfrak m(a)=r$ if and only if $a\in \mathfrak m^r\setminus \mathfrak m^{r+1}$, and we set $\ord_\mathfrak m(0)=+\infty$.

For a element $a\in R$, if $a\in\mathfrak m^r\setminus\mathfrak m^{r+1}$, its initial form is
$$
\mathrm{in}(a)=a+\mathfrak m^{r+1}\in \mathfrak m^r/\mathfrak m^{r+1},
$$
where $\mathrm{ord}_m(a)=r$.

Note that if $r=\ord_{\m}(a)$, then $\mathrm{in}(a)=[a]\in \m^{r}/\m^{r+1}$ is nonzero. However, if $i<\ord_{\m}(a)$, then $[a]=0$ in $\m^{i}/\m^{i+1}$.

For an ideal $\mathfrak a\subseteq R$, the initial ideal is the homogeneous ideal
$$
\mathrm{in}(\mathfrak a)=(\,\mathrm{in}(a)\mid a\in \mathfrak a\,)\subseteq \mathrm{gr}_{\mathfrak m}(R).
$$
In particular, if $\mathfrak a=(x_1,\dots,x_n)$, then $\mathrm{in}(\mathfrak a)=(\mathrm{in}(x_1),\dots,\mathrm{in}(x_n))$.
\end{definition}

A comparison of the Frobenius invariants of a local ring and its associated graded ring requires an understanding of the behavior of ideals in the associated graded setting. The remark below makes this precise by introducing the initial ideal.

\begin{remark}
Let $(R, \mathfrak{m}, k)$ be a local ring, and let $A = \mathrm{gr}_{\mathfrak{m}}(R) = \bigoplus_{i=0}^\infty \mathfrak{m}^i / \mathfrak{m}^{i+1}$ be its associated graded ring. Let $\mathfrak{a} \subseteq R$ be an ideal. The initial ideal of $\mathfrak{a}$ in $A$ is defined as
\[
\mathrm{gr}_{\mathfrak{m}}(\mathfrak{a}) = \bigoplus_{i=0}^\infty \frac{\mathfrak{a} \cap \mathfrak{m}^i}{\mathfrak{a} \cap \mathfrak{m}^{i+1}}.
\]
For each $i \geq 0$, define two $k$-modules:
\[
V_i = \frac{\mathfrak{a} \cap \mathfrak{m}^i}{\mathfrak{a} \cap \mathfrak{m}^{i+1}}, \quad 
W_i = \frac{\mathfrak{a} \cap \mathfrak{m}^i + \mathfrak{m}^{i+1}}{\mathfrak{m}^{i+1}} \subseteq \frac{\mathfrak{m}^i}{\mathfrak{m}^{i+1}}.
\]\[
\phi_i: \mathfrak{a} \cap \mathfrak{m}^i \to \frac{\mathfrak{a} \cap \mathfrak{m}^i + \mathfrak{m}^{i+1}}{\mathfrak{m}^{i+1}}, \quad x \mapsto x + \mathfrak{m}^{i+1}.
\]
Note that $\phi_i$ is surjective with kernel $\mathfrak{a} \cap \mathfrak{m}^{i+1}.$
Hence, by the First Isomorphism Theorem, we have
\[
\frac{\mathfrak{a} \cap \mathfrak{m}^i}{\mathfrak{a} \cap \mathfrak{m}^{i+1}} \cong \frac{\mathfrak{a} \cap \mathfrak{m}^i + \mathfrak{m}^{i+1}}{\mathfrak{m}^{i+1}},
\]
so $V_i \cong W_i$ for all $i$. Therefore,
\[
\mathrm{gr}_\mathfrak{m}(\mathfrak{a}) = \bigoplus_{i=0}^\infty V_i \cong \bigoplus_{i=0}^\infty W_i = \bigoplus_{i=0}^\infty \frac{\mathfrak{a} \cap \mathfrak{m}^i + \mathfrak{m}^{i+1}}{\mathfrak{m}^{i+1}}\cong \bigoplus_{i=0}^\infty \frac{\mathfrak{a} \cap \mathfrak{m}^i + \mathfrak{m}^{i+1}}{\mathfrak{m}^{i+1}}.
\]
We note that $\mathrm{gr}_\mathfrak{m}(\mathfrak{a})$ is an ideal in $\mathrm{gr}_\mathfrak{m}(R)$.

\end{remark}

We record the following standard identification for later use.

\begin{remark}

Let $(R,\mathfrak m,k)$ be a local ring, and let $\mathfrak a\subseteq R$ be an ideal. For $x\in \mathfrak a$, let $\mathrm{ord}_{\mathfrak m}(x)=i$. We claim that
$
\mathrm{in}(\mathfrak a)=\mathrm{gr}_{\mathfrak m}(\mathfrak a)
$ as graded submodules of $\mathrm{gr}_{\mathfrak m}(R)$. Fix $i\geq 0$. The degree-$i$ component of $\mathrm{in}(\mathfrak a)$ is given by
$$
\big [\mathrm{in}(\mathfrak a)\big ]_i \;=\; \langle \mathrm{in}(x)\in \mathfrak m^i/\mathfrak m^{i+1} \mid x\in \mathfrak a,\ \mathrm{ord}_{\mathfrak m}(x)=i\rangle.
$$

On the other hand, consider the natural map
$$
\phi_i:\ \frac{\mathfrak a\cap \mathfrak m^i}{\mathfrak a\cap \mathfrak m^{i+1}}\longrightarrow \frac{\mathfrak m^i}{\mathfrak m^{i+1}},\qquad
y+(\mathfrak a\cap \mathfrak m^{i+1})\mapsto y+\mathfrak m^{i+1}.
$$
The image of this map consists of all classes $y+\mathfrak m^{i+1}$ with $y\in \mathfrak a\cap \mathfrak m^i$. Such a class is zero precisely when $y\in \mathfrak a\cap \mathfrak m^{i+1}$. Therefore, the nonzero elements of $\mathrm{im}(\phi_i)$ are represented by those $y\in \mathfrak a$ with $\mathrm{ord}_{\mathfrak m}(y)=i$, in which case $y+\mathfrak m^{i+1}=\mathrm{in}(y)$. It follows that
$$
\mathrm{im}(\phi_i)=\big [\mathrm{in}(\mathfrak a)\big ]_i.
$$

By definition, the source of $\phi_i$ is $(\mathfrak a\cap \mathfrak m^i)/(\mathfrak a\cap \mathfrak m^{i+1})$, which is the degree-$i$ component of $\mathrm{gr}_{\mathfrak m}(\mathfrak a)$. Since $\phi_i$ is injective into $\mathfrak m^i/\mathfrak m^{i+1}$, we obtain
$$
\big[\mathrm{in}(\mathfrak a)\big]_i=\mathrm{im}(\phi_i)=\frac{\mathfrak a\cap \mathfrak m^i}{\mathfrak a\cap \mathfrak m^{i+1}}=\big[\mathrm{gr}_{\mathfrak m}(\mathfrak a)\big]_i.
$$
Therefore, the degree-$i$ components of $\mathrm{in}(\mathfrak a)$ and $\mathrm{gr}_{\mathfrak m}(\mathfrak a)$ coincide for all $i\geq 0$. Consequently,
$$
\mathrm{in}(\mathfrak a)=\mathrm{gr}_{\mathfrak m}(\mathfrak a)\subseteq \mathrm{gr}_{\mathfrak m}(R).
$$
\end{remark}

\section{Behavior of $F$-thresholds with Respect to Fiber of Blow-ups}\label{F-thresholds}
In this section, we prove \autoref{20}, giving a precise comparison between the $F$-thresholds of a local ring $(R,\mathfrak{m})$ and those of its associated graded ring $\mathrm{gr}_{\mathfrak{m}}(R)$. 

 \begin{theorem}\label{20}
 Let $(R, \mathfrak{m}, k)$ be a local $F$-finite ring of prime characteristic $p>0$, and $\mathfrak{b}\subseteq R$ be an $\mathfrak{m}$-primary ideal. Let $A = \mathrm{gr}_{\mathfrak{m}}(R)$ be its associated graded ring with maximal homogeneous ideal $\mathfrak{n}$. Then,
$$c^{\mathfrak{b}}(\mathfrak{m}) \leq c^{I}(\mathfrak{n}),$$
where $I=\mathrm{gr}_\m(\b)$.
\end{theorem}

\begin{proof} 
We first assume that $R$ is complete. If $R$ is not complete, the result follows from 
$$c^\mathfrak{b}(\mathfrak{m})=c^{\mathfrak{b}\hat{R}}(\mathfrak{m}\hat{R})\leq c^I(\mathfrak{n}),$$
because $\mathrm{gr}_\mathfrak{m}(R)=\mathrm{gr}_{\mathfrak{m}\hat{R}}(\hat{R})$.  We pick $x_i\in \b$ such that $I = (\mathrm{in}(x_1), \ldots, \mathrm{in}(x_s))$. Then, 
$\mathfrak{b}= (x_1, \ldots, x_s)$. We have that $I^{[p^e]} = (\mathrm{in}(x_1)^{p^e}, \ldots, \mathrm{in}(x_s)^{p^e})$. Suppose $t$  is such that $\mathfrak{n}^t \subseteq {I}^{[p^e]},$ and let $f \in \mathfrak{m}^t$. Since $\mathfrak{n}^t \subseteq {I}^{[p^e]}$, we have
\[
[f] = \sum_{i=1}^s \mathrm{in}(r_{i,0}) \mathrm{in}(x_i)^{p^e} \in \mathfrak{m}^t / \mathfrak{m}^{t+1},
\]
where $ x_i \in \mathfrak{b} \cap \mathfrak{m}^{d_i} $, with $d_i = \mathrm{ord}_{\mathfrak m}(x_i)$. Then, $$ \mathrm{in}(x_i) = x_i + \mathfrak{m}^{d_i + 1} \in (\mathfrak{b} \cap \mathfrak{m}^{d_i}) / (\mathfrak{b} \cap \mathfrak{m}^{d_i + 1}) \subseteq I, $$ and $$ \mathrm{in}(r_{i,0}) = r_{i,0} + \mathfrak{m}^{t - d_i p^e + 1} \in \mathfrak{m}^{t - d_i p^e} / \mathfrak{m}^{t - d_i p^e + 1}.$$ The degree condition ensures that the sum is homogeneous of degree $ t $, because  $(t - d_i p^e) + d_i p^e = t.$ This implies
\[
f = \sum_{i=1}^s r_{i,0} x_i^{p^e} + s_1, \quad s_1 \in \mathfrak{m}^{t+1}.
\]
Since $ [s_1] \in \mathfrak{m}^{t+1} $, its initial form is $$\mathrm{in}(s_1) = s_1 + \mathfrak{m}^{t+2} \in \mathfrak{m}^{t+1} / \mathfrak{m}^{t+2} \subseteq \mathfrak{n}^{t+1}. $$ Then, we have $ \mathfrak{n}^{t+1} \subseteq \mathfrak{n}^t \subseteq I^{[p^e]} $, so
\[
[s_1] = \sum_{i=1}^s \mathrm{in}(r_{i,1}) \mathrm{in}(x_i)^{p^e},
\]
where 
\[ \mathrm{in}(r_{i,1}) = r_{i,1} + \mathfrak{m}^{t + 1 - d_i p^e + 1} \in \mathfrak{m}^{t + 1 - d_i p^e} / \mathfrak{m}^{t + 1 - d_i p^e + 1}, \]
and the degree is $ (t + 1 - d_i p^e) + d_i p^e = t + 1 $. We have
\[
s_1 = \sum_{i=1}^s r_{i,1} x_i^{p^e} + s_2,
\]
with $ [s_2] \in \mathfrak{m}^{t+2} $, and $ \sum_{i=1}^s r_{i,1} x_i^{p^e} \in \mathfrak{b}^{[p^e]} $. Substituting back
\[
f = \sum_{i=1}^s r_{i,0} x_i^{p^e} + \sum_{i=1}^s r_{i,1} x_i^{p^e} + s_2.
\]

We proceed inductively. For $ [s_k] \in \mathfrak{m}^{t+k} $, we have $$s_k + \mathfrak{m}^{t+k+1} \in \mathfrak{m}^{t+k} / \mathfrak{m}^{t+k+1} \subseteq \mathfrak{n}^{t+k} \subseteq \mathfrak{n}^t \subseteq I^{[p^e]},$$ so
\[
[s_k]= \sum_{i=1}^s \mathrm{in}(r_{i,k}) \mathrm{in}(x_i)^{p^e},
\]
with $ \mathrm{in}(r_{i,k}) \in \mathfrak{m}^{t + k - d_i p^e} / \mathfrak{m}^{t + k - d_i p^e + 1} $, and in $ R $
\[
s_k = \sum_{i=1}^s r_{i,k} x_i^{p^e} + s_{k+1}, \quad s_{k+1} \in \mathfrak{m}^{t+k+1}.
\]
Thus, for any $ n$, we have
\[
f = \sum_{k=0}^n \sum_{i=1}^s r_{i,k} x_i^{p^e} + s_{n+1}, \quad s_{n+1} \in \mathfrak{m}^{t+n+1}.
\]
For each $ i $, the sum $ \sum_{k=0}^\infty r_{i,k} $ converges to some  $ r_i \in R $, because  $r_{i,k} \in \mathfrak{m}^{t + k - d_i p^e}$  and $R$ is complete. Thus
\[
f = \sum_{i=1}^s \left( \sum_{k=0}^\infty r_{i,k} \right) x_i^{p^e} = \sum_{i=1}^s r_i x_i^{p^e}.
\]
Since each $ x_i^{p^e} \in \mathfrak{b}^{[p^e]} $, we have $ f \in \mathfrak{b}^{[p^e]} $. As $ f \in \mathfrak{m}^t $ was arbitrary, we conclude $\mathfrak{m}^t \subseteq \mathfrak{b}^{[p^e]}.$ Then, $\min\{t\in\mathbb{N}\mid \mathfrak{m}^t\not\subseteq \mathfrak{b}^{[p^e]}\}\leq \min\{t\in\mathbb{N}\mid \mathfrak{n}^t\not\subseteq I^{[p^e]}\}.$ Therefore, for every $e$,
$$\nu^{\mathfrak{b}}_{\mathfrak{m}}(p^e) \leq \nu^{I}_{\mathfrak{n}}(p^e).$$ Dividing by $p^e$ and then taking limit, we have 
\[
c^{\mathfrak{b}}(\mathfrak{m}) \leq c^{I}(\mathfrak{n}).
\]

\end{proof}

The next corollary follows immediately from the preceding theorem.

 \begin{corollary}
 Let $(R, \mathfrak{m}, k)$ be a local $F$-finite ring of prime characteristic $p>0$, and let $A = \mathrm{gr}_{\mathfrak{m}}(R)$ be its associated graded ring with maximal homogeneous ideal $\mathfrak{n}$. Then,
$$c^{\mathfrak{m}}(\mathfrak{m}) \leq c^{\mathfrak{n}}(\mathfrak{n}).$$
\end{corollary}

The inequality in the preceding lemma need not be an equality, as the following example shows.

\begin{example}
Let $k$ be a field of characteristic $p>0$ and set
\[
R=\frac{k[[x,y,z,w]]}{(xy-z^{2}w)},\qquad \mathfrak m=(x,y,z,w).
\]
Since $\ord_{\mathfrak m}(xy)=2<\ord_{\mathfrak m}(z^{2}w)=3$, 
\[
\mathrm{gr}_{\mathfrak m}(R)\cong \frac{k[x,y,z,w]}{(xy)},
\]
whose homogeneous maximal ideal we denote by $\mathfrak n$.
We have $c^{\mathfrak m}(\mathfrak m)=\tfrac{5}{2}$ \cite[Theorem 2]{MOY}. 
Moreover, $\mathrm{gr}_{\mathfrak m}(R)$ is a $3$-dimensional standard graded ring, so
$c^{\mathfrak n}(\mathfrak n)=3$. Consequently,
\[
c^{\mathfrak m}(\mathfrak m)=\frac{5}{2}\;<\;3=c^{\mathfrak n}(\mathfrak n).
\]
\end{example}

When the ring $(R,\mathfrak{m},k)$ is regular—whether we are working in the local case or in the standard graded case—the relationship between $F$-pure thresholds and $F$-thresholds is as nice as one could hope: for every ideal $\mathfrak{a} \subseteq R$ one has an exact equality $
\mathrm{fpt}(\mathfrak{a}) = c^{\mathfrak{m}}(\mathfrak{a}).$

This correspondence, however, begins to break down as soon as singularities appear. In fact, in general one can only guarantee the weaker inequality $
\mathrm{fpt}(\mathfrak{a}) \leq c^{\mathfrak{m}}(\mathfrak{a}),$ so equality becomes a very special phenomenon.
This has been investigated from different angles. In the graded setting, De Stefani, Núñez-Betancourt, and Pérez \cite{DSNBP} showed that if one starts with a standard graded Gorenstein $k$-algebra that is $F$-finite and $F$-pure, then the condition $\mathrm{fpt}(\mathfrak{m}) = c^{\mathfrak{m}}(\mathfrak{m})
$ already forces the ring to be strongly $F$-regular. In other words, within this class of rings, achieving equality means that the ring has the strongest and most desirable type of $F$-singularity.

De Stefani, Núñez-Betancourt, and Smirnov \cite{DNS} showed that if $(R,\mathfrak{m})$ is an $F$-finite local ring whose associated graded ring $\mathrm{gr}_{\mathfrak{m}}(R)$ is Gorenstein, then the behavior of $F$-singularities in the graded ring reflects back to the local ring itself. More concretely, if the associated graded ring is $F$-pure, then one has an equality of thresholds, $\mathrm{fpt}(R) = \mathrm{fpt}(\mathrm{gr}_{\mathfrak{m}}(R))$, which in particular shows that $R$ is $F$-pure. Even more, if the associated graded ring is strongly $F$-regular, then the local ring automatically inherits strong $F$-regularity.

\begin{corollary}
Let $(R, \mathfrak{m}, k)$ be an $F$-finite local ring of characteristic $p > 0$. Let $A = \mathrm{gr}_{\mathfrak{m}}(R)$ be the associated graded ring with maximal homogeneous ideal $\mathfrak{n}$. Suppose that $A$ is Gorenstein. If $A$ is $F$-pure and $\mathrm{fpt}(A) = c^{\mathfrak{n}}(\mathfrak{n})$, then $\mathrm{fpt}(R) = c^{\mathfrak{m}}(\mathfrak{m})$ and $R$ is strongly $F$-regular.

\end{corollary}
\begin{proof} We have that
\[
\mathrm{fpt}(R) \leq c^{\mathfrak{m}}(\mathfrak{m}) \leq c^{\mathfrak{n}}(\mathfrak{n})=\mathrm{fpt}(A)=\mathrm{fpt}(R).
\] Hence, $\mathrm{fpt}(R) = c^{\mathfrak{m}}(\mathfrak{m})$. Since $A$ is strongly $F$-regular \cite[Theorem B]{DSNBP}, so is $R$ \cite[Theorem 5.8]{DNS}.

\end{proof}
\section{Behavior of $F$-rationality with Respect to Fiber of Blow-ups}\label{F-rationality}
The concept of $F$-rationality plays a central role in the study of $F$-singularities, as it characterizes mild singular behavior through the tight closure of parameter ideals.  Our main result in this section shows that if the associated graded ring $\mathrm{gr}_{\mathfrak{m}}(R)$ is $F$-rational, then the local ring $R$ is $F$-rational as well. The result follows from inequality in \autoref{20}, which transfers the key bound from the associated graded ring to the original ring; together with the characterization of $F$-thresholds  in \autoref{7}, this yields the implication from $\mathrm{gr}_{\m}(R)$ to $R$.

In preparation for the next proposition, we review the notion of superficial elements, which are frequently used in dimension theory and in the analysis of associated graded rings.

\begin{definition}

 An element $ a \in \mathfrak{m} \setminus \mathfrak{m}^2 $ is superficial for maximal ideal $ \mathfrak{m} $ if there exists a non-negative integer $ c $ such that
  \[
  (\mathfrak{m}^{n+1} : a) \cap \mathfrak{m}^c = \mathfrak{m}^n 
  \]
 for all $n \geq c.$
 \end{definition}

This condition ensures that multiplication by $a$ preserves the structure of the powers of $\mathfrak{m}$ in a controlled way. Such elements are particularly useful in reducing the dimension of a ring modulo $a$, a feature that becomes evident in the following proposition:

\begin{proposition}[{\cite[p.296]{ZS}}]\label{13}
Let $(R,\mathfrak{m},k)$ be a $d$-dimensional local ring and $a$ be a superficial element in $\mathfrak{m}$. Let $A=gr_\mathfrak{m}(R)$ be its associated graded ring. Then, \begin{enumerate}
\item  $ \mathrm{in}(a) \notin \bigcup_{i=1}^{m} Q_i $, where $ Q_i \in \mathrm{Ass}(A) $, $ Q_i \neq \bigoplus_{i>0} \mathfrak{m}^i / \mathfrak{m}^{i+1} $.
\item $ \dim R / aR = d - 1 $.

  \item $a$ is $ R $-regular if and only if $ \mathrm{depth} \,R > 0 $.
 \item For $ j \geq 1 $, $ \mathrm{depth} \,R / aR \geq j $ if and only if $ \mathrm{depth} \, R \geq j + 1 $.
  \item $\mathrm{in}(a)$ is $ A $-regular if and only if $ \mathrm{depth} \,A > 0.$

\end{enumerate}
\end{proposition}

The following lemma, commonly known as Sally’s machine, offers a method for transferring depth conditions from the quotient ring $R/(a)$ to $R$. It is a standard tool in inductive arguments concerning depth.

\begin{lemma}[{Sally’s machine \cite[Lemma 2.2]{HM}}] Let $(R,\mathfrak{m},k)$ be a $d$-dimensional local ring and $a$ be a superficial element in $\mathfrak{m}$. For $ j \geq 1 $, $ \mathrm{depth} \,\mathrm{gr}_{\mathfrak{m}/( a)}(R/(a)) \geq j $ if and only if $ \mathrm{depth}\, \mathrm{gr}_{\mathfrak{m}}(R) \geq j + 1 $.
\end{lemma}

It is expected that if $\mathrm{gr}_{\mathfrak{m}}(R)$ satisfies a good property, then $R$ also satisfies the property. 
 The following lemma make this relationship for Cohen--Macaulayess. We point out that this known to the experts, and we add it for the sake of completeness.

\begin{lemma}\label{2}
Let $ (R, \mathfrak{m}, k)$ be a $d$-dimensional Noetherian local ring. Let  $A = \mathrm{gr}_{\mathfrak{m}}(R)$ be the associated graded ring with maximal homogeneous ideal $\mathfrak{n}$. Assume $A$ is Cohen--Macaulay. Then, $R$ is Cohen--Macaulay.
\end{lemma}

\begin{proof}

We first reduce to the case where the residue field $k = R/\mathfrak{m}$ is infinite. If $k$ is already infinite, we proceed directly. Otherwise, we replace $R$ by the local ring
$$
R'= R[X]_{\mathfrak{m} R[X]},
$$
where $X$ is an indeterminate. The extension $R \to R'$ is faithfully flat and local, with maximal ideal $\mathfrak{m}R'$. The residue field of $R'$ is $k(X)$, which is infinite. 

Moreover, the associated graded ring satisfies
$$
\mathrm{gr}_{\mathfrak{m} R'}(R') \cong \mathrm{gr}_{\mathfrak{m}}(R) \otimes_k k(X) = A \otimes_k k(X),
$$
which is Cohen--Macaulay as $A$ is Cohen--Macaulay and base change by field extensions preserves the Cohen--Macaulay property. Therefore, it suffices to prove the result for $R'$, and then descend the conclusion to $R$ via faithful flatness. Thus, we may assume without loss of generality that the residue field $k$ is infinite.

Let $J=(\mathrm{in}(x_1),\ldots, \mathrm{in}(x_d))$ be an ideal of $A$ generated by a full system of parameters. We proceed by induction on $d = \dim R$. If $d = 0$,  both  $A = \mathrm{gr}_{\mathfrak{m}}(R)$  and $R$ are Cohen--Macaulay. 

Assume now that $d \geq 1$, and suppose the result holds in dimension $d-1$. Since $A = \mathrm{gr}_{\mathfrak{m}}(R)$ is Cohen--Macaulay of dimension $d$, and $k$ is infinite, we can choose an element $x_1 \in \mathfrak{m} \setminus \mathfrak{m}^2$ such that its initial form $\mathrm{in}(x_1) \in \mathfrak{m}/\mathfrak{m}^2 \subseteq A$ is a nonzerodivisor on $A$. Then $\mathrm{in}(x_1)$ is $A$-regular, and since $A$ is Cohen--Macaulay, We have
$A / (\mathrm{in}(x_1)) \cong \mathrm{gr}_{\mathfrak{m}/(x_1)}(R/(x_1))$
\cite[Exercise 5.3]{Eisenbud},
and the quotient ring is Cohen--Macaulay of dimension $d-1$. By the induction hypothesis, $R/(x_1)$ is Cohen--Macaulay.  By \autoref{13}(4), $\mathrm{depth}(R) = d$. Therefore, $R$ is Cohen--Macaulay. 

\end{proof}

The next lemma establishes a basic connection between the associated graded ring and the local ring. It shows that if the initial form of an element is a nonzerodivisor on the graded ring, then the element itself is a nonzerodivisor and the expected colon relation holds.

\begin{lemma}\label{colon}
Let $(R,\m, k)$ be a Noetherian local ring and let $A=\mathrm{gr}_\m(R)$. Let $x\in \m\setminus \m^2$ and suppose that $\mathrm{in}(x)\in A_1$ is a nonzerodivisor on $A$. 
Then, 
\[
(\m^{\,n+1} :_R x) = \m^n \quad \forall\, n\ge 0.
\]
Moreover, $x$ is a nonzerodivisor on $R$.
\end{lemma}

\begin{proof}
The inclusion $\m^n\subseteq(\m^{n+1}:x)$ follows from  $x\in\m$. For the reverse, let $f \in (\m^{n+1}:x)$. That means $xf \in \m^{n+1}$.
Suppose for contradiction that $f \notin \m^n$. Let $t = \mathrm{ord}_\m(f)$, then by assumption $t < n-1$. Since $f \in \m^t \setminus \m^{t+1}$, its initial form is
$$
\mathrm{in}(f) = f + \m^{t+1} \;\in\; \m^t/\m^{t+1}.
$$
Because $x \in \m\setminus\m^2$, $\mathrm{in}(x)$ has degree $1$. Therefore
$$
\mathrm{in}(xf) = \mathrm{in}(x)\cdot \mathrm{in}(f) \;\in\; \m^{t+1}/\m^{t+2}.
$$
We also know $xf \in \m^{n+1}$. Since $t < n-1$, this forces $n \ge t+1$. Thus $xf \in \m^{n+1} \subseteq \m^{t+2}.$ Hence in degree $t+1$ we have
$$
  \mathrm{in}(xf) = xf + \m^{t+2} = 0 + \m^{t+2}.
  $$
Therefore, $$
\mathrm{in}(x)\cdot \mathrm{in}(f) = \mathrm{in}(xf) = 0 \quad \text{in } \m^{t+1}/\m^{t+2}.
$$

Since $\mathrm{in}(x)$ is a nonzerodivisor in $A$, this forces $\mathrm{in}(f)=0$, i.e. $f\in \m^{t+1}$. Contradicting the choice of $t$. Thus no such $f$ exists, and we conclude $$
(\m^{n+1}:x)\subseteq \m^n.
$$
As a consequence, $x$ is a regular element in $R$ \cite[Lemma 8.5.3]{HS}.

\end{proof}

Building on \autoref{colon}, we now examine how systems of parameters in the associated graded ring reflect back to the local ring. When the graded ring is Cohen--Macaulay, the next lemma shows that a degree-one system of parameters in the graded ring lifts to a full system of parameters in the local ring.

\begin{lemma}\label{3}
Let $ (R, \mathfrak{m}, k)$ be a $d$-dimensional Noetherian local ring with infinite residue field. Let  $A = \mathrm{gr}_{\mathfrak{m}}(R)$ be the associated graded ring with maximal homogeneous ideal $\mathfrak{n}$. Assume that $A$ is Cohen--Macaulay.  Let $J=(\mathrm{in}(x_1),\ldots, \mathrm{in}(x_d))$ be an ideal of $A$ generated by a full system of parameters of degree $1$. Let $\mathfrak{a}=(x_1,\ldots,x_d)$ be an ideal of $R$. Then, $\mathfrak{a}$ is generated by a system of parameters.

\end{lemma}

\begin{proof}
 Since $\mathrm{in}(x_1)\in \left [A\right ]_{1}$ is regular, by \autoref{colon}, we have  $x_1$ is a nonzerodivisor on $R$. Thus $\dim R/(x_1)=d-1$. Therefore,  we have a graded isomorphism
$\mathrm{gr}_{\m/(x_1)}(R/(x_1))\ \cong\ A/(\mathrm{in}(x_1))$
\cite[Exercise 5.3]{Eisenbud},
and the right-hand side is Cohen--Macaulay of dimension $d-1$. Moreover, the images of $\mathrm{in}(x_2),\dots,\mathrm{in}(x_d)$ form a (degree-$1$) system of parameters in $A/(\mathrm{in}(x_1))$. By the induction hypothesis applied to $R/(x_1)$, the images $\bar x_2,\dots,\bar x_d$ form a system of parameters there. Hence $\dim R/(x_1,\dots,x_d)=0$, so $\a$ is $\m$-primary and $x_1,\ldots,x_d$ is a system of parameters of $R$.

\end{proof}

The next step is to refine our earlier results by highlighting the connection between integral closures and the associated graded ring. When the graded ring is Cohen--Macaulay, the proposition below demonstrates that the integral closure of specific parameter ideals is precisely the maximal ideal.

\begin{proposition}\label{6}
Let $ (R, \mathfrak{m}, k)$ be a $d$-dimensional Noetherian excellent local ring with infinite residue field. Let  $A = \mathrm{gr}_{\mathfrak{m}}(R)$ be the associated graded ring with maximal homogeneous ideal $\mathfrak{n}$. Assume that $A$ is Cohen--Macaulay. Let $J=(\mathrm{in}(x_1),\ldots, \mathrm{in}(x_d))$ be an ideal of $A$ generated by a full system of parameters of degree $1$ such that $\bar{J}=\mathfrak{n}$. Let $\mathfrak{a}=(x_1,\ldots,x_d)$ be an ideal of $R$. Then, $\bar{\mathfrak{a}}=\mathfrak{m}.$
\end{proposition}

\begin{proof}
Note that, using \autoref{2} and \autoref{3}, $R$ is Cohen--Macaulay and $x_1, \ldots, x_d$ form a regular sequence in $R$. The condition $\bar{J} = \mathfrak{n}$ implies that $J$ is a minimal reduction of $\mathfrak{n}$ in $A$. Therefore, there exists $n_0$ such that for all $n \geq n_0$,
\[
\mathfrak{n}^{n+1} = J \mathfrak{n}^n.
\]
We need to show that $\mathfrak{a} = (x_1, \dots, x_d)$ is a reduction of $\mathfrak{m}$ in $R$,
\[
\mathfrak{m}^{n+1} = \mathfrak{a} \mathfrak{m}^n, \quad \text{for } n \gg 0.
\]

Let $r \in \mathfrak{m}^{n+1}$. Then its image in the associated graded ring is $[r]:= r+\mathfrak{m}^{n+2}\in \mathfrak{m}^{n+1} / \mathfrak{m}^{n+2} \subseteq \mathfrak{n}^{n+1}$. Since $\mathfrak{n}^{n+1} = J \mathfrak{n}^n$ for $n \geq n_0$, we can write
\[
[r] = \sum_{i=1}^d \mathrm{in}(x_i) [g_i], \quad [g_i] \in \mathfrak{n}^n = \mathfrak{m}^n / \mathfrak{m}^{n+1}.
\]
where $g_i \in \mathfrak{m}^n$ such that $[g_i] = g_i + \mathfrak{m}^{n+1}$. Then,
\[
r - \sum_{i=1}^d x_i g_i \in \mathfrak{m}^{n+2}.
\]
Thus
\[
r \in \mathfrak{a} \mathfrak{m}^n + \mathfrak{m}^{n+2}.
\]
This implies
\[
\mathfrak{m}^{n+1} \subseteq \mathfrak{a} \mathfrak{m}^n + \mathfrak{m}^{n+2}.
\]
Since $\mathfrak{a}\subseteq \mathfrak{m}$, $\mathfrak{a}\mathfrak{m}^n\subseteq \mathfrak{m}^{n+1}$. We have $\mathfrak{m}^{n+2}\subseteq \mathfrak{m}^{n+1}$. Therefore, 
\[
\mathfrak{m}^{n+1} = \mathfrak{a} \mathfrak{m}^n + \mathfrak{m}^{n+2}.
\]
By Nakayama’s Lemma,  
\[
\mathfrak{m}^{n+1} = \mathfrak{a} \mathfrak{m}^n \quad \text{for } n \gg 0.
\]
This shows that $\mathfrak{a}$ is a reduction of $\mathfrak{m}$. Thus, $\overline{\mathfrak{a}} = \mathfrak{m}.$

\end{proof}

With the behavior of integral closures in place, our next step is to clarify the exact correspondence between an ideal in $R$ and its graded analog in the associated graded ring.

\begin{proposition}\label{PropGenParameterIdeal}
Let $ (R, \mathfrak{m}, k)$ be a $d$-dimensional excellent  Noetherian local ring with infinite residue field. Let  $A = \mathrm{gr}_{\mathfrak{m}}(R)$ be the associated graded ring with maximal homogeneous ideal $\mathfrak{n}$. Suppose that $A$ is Cohen--Macaulay. Let $J=(\mathrm{in}(x_1),\ldots, \mathrm{in}(x_d))$ be an ideal of $A$ generated by a full system of parameters of degree $1$. Let $\mathfrak{a}=(x_1,\ldots,x_d)$ be an ideal of $R$. Then, $\mathrm{gr}_{\mathfrak{m}}(\mathfrak{a}) = J.$
\end{proposition}

\begin{proof}
We proceed by induction on $d = \dim R$. If $d=1$, $\mathrm{in}(x_1)$ is regular in $A$. Then, using \autoref{colon}, $x_1$ is a nonzerodivisor in $R$. Since $(x_1)$ is principal, we have that $\mathrm{gr}_{\mathfrak m}((x_1)) = (\mathrm{in}(x_1))$. Thus the claim holds.

Assume the result holds for all rings of dimension $d-1$.
Let $R$ have dimension $d$. It follows from \autoref{3} that $\mathfrak a = (x_1,\dots,x_d)$ is generated by a system of parameters .
Since $\mathrm{in}(x_1)\in A$ is regular, using \autoref{colon}, $x_1$ is a nonzerodivisor on $R$. Let $\tilde{R}= R/(x_1)$ and $\tilde{\m}= \mathfrak m/(x_1)$.
We have that  $\tilde{R}$ is Cohen--Macaulay of dimension $d-1$. We have that
$\mathrm{gr}_{\tilde{\m}}(\tilde{R}) \cong \mathrm{gr}_{\mathfrak m/(x_1)}(R/(x_1)) \cong A / (\mathrm{in}(x_1))$ \cite[Exercise 5.3]{Eisenbud}.

Sally’s machine  gives
$$
\mathrm{depth}\,\mathrm{gr}_{\tilde{\m}}(\tilde{R}) \ge d-1,
$$
so $\mathrm{gr}_{\tilde{\m}}(\tilde{R})$ is Cohen--Macaulay of dimension $d-1$. Let $\bar{x}_i$ denote the image of $x_i$ in $\tilde{R}$ for $i=2,\dots,d$.
Then $(\mathrm{in}(\bar{x}_2),\dots,\mathrm{in}(\bar{x}_d))$ is a parameter ideal of $\mathrm{gr}_{\tilde{\m}}(\tilde{R})$.
By the induction hypothesis applied to $\tilde{R}$ and $\bar{\mathfrak {a}} := (\bar{x}_2,\dots,\bar{x}_d)$, we have
$$
\mathrm{gr}_{\tilde{\m}}(\bar{\mathfrak a}) = (\mathrm{in}(\bar{x}_2),\dots,\mathrm{in}(\bar{x}_d)).
$$

Lifting back to $A$, these are the images of $\mathrm{in}(x_2),\dots,\mathrm{in}(x_d)$ modulo $\mathrm{in}(x_1)$. Indeed, we have the canonical graded surjection 
$$
\varphi:\; A=\operatorname{gr}_{\mathfrak m}(R)\longrightarrow \operatorname{gr}_{\tilde{\mathfrak m}}(\tilde R)
$$
induced by the quotient $R\to\tilde R=R/(x_1)$. Since $A/(\operatorname{in}(x_1))\cong \operatorname{gr}_{\tilde \m}(\tilde R)$, we have $\ker \varphi=(\operatorname{in}(x_1))$. On the $l$-th graded piece, $\varphi$ is given by
$$
r+\m^{\,l+1}\;\longmapsto\;(r+(x_1))+\tilde \m^{\,l+1}, 
$$
where $l=\mathrm{ord}_\m(r)$. In particular, if $x_i\in \m^l\setminus \m^{\,l+1}$, then

$$
\varphi(\operatorname{in}(x_i))=\varphi(x_i+\m^{\,l+1})=(x_i+(x_1))+\tilde \m^{\,l+1}=\bar{x}_i+\tilde \m^{\,l+1}=\operatorname{in}(\bar{x}_i).
$$
Note that for $i=1$, both sides vanish. Therefore, $
\varphi(\operatorname{in}(x_i))=\operatorname{in}(\bar{x}_i)$ for all $i$. It follows that $
\varphi(\operatorname{gr}_{\m}(\a))\;\subseteq\;\operatorname{gr}_{\tilde \m}(\bar \a).
$

By the induction hypothesis,
$$
\varphi(\mathrm{gr}_{\mathfrak m}(\mathfrak a))\subseteq \big(\operatorname{in}(\bar x_2),\dots,\operatorname{in}(\bar x_d)\big)=\varphi\big((\operatorname{in}(x_2),\dots,\operatorname{in}(x_d))\big)=\varphi(J).
$$
Hence, $\varphi(\mathrm{gr}_{\mathfrak m}(\mathfrak a))\subseteq \varphi(J)$. Then,
$$
\mathrm{gr}_{\mathfrak m}(\mathfrak a)\subseteq J+\ker\varphi=J+(\operatorname{in}(x_1)).
$$
However, $J$ already contains $\operatorname{in}(x_1)$, so $J+(\operatorname{in}(x_1))=J$. Therefore, $\mathrm{gr}_{\mathfrak m}(\mathfrak a)\subseteq J$.

The reverse inclusion $J\subseteq \mathrm{gr}_{\mathfrak m}(\mathfrak a)$ is immediate because each $\operatorname{in}(x_i)$ for all $i$ lies in $\operatorname{gr}_{\mathfrak m}(\mathfrak a)$ by definition. Thus, $J\subseteq \mathrm{gr}_{\mathfrak m}(\mathfrak a)$. It follows that 
$$\operatorname{gr}_{\mathfrak m}(\mathfrak a)=(\operatorname{in}(x_1),\dots,\operatorname{in}(x_d))=J,
$$
as required.

\end{proof}

The lemma below establishes that when two associated graded ideals agree, the corresponding ideals in the ring are in fact identical.

\begin{lemma}\label{22}

Let $(R,\mathfrak m, k)$ be a Noetherian local ring and $\mathfrak a\subseteq \mathfrak b$ ideals.
If $\mathrm{gr}_{\mathfrak m}(\mathfrak a)=\mathrm{gr}_{\mathfrak m}(\mathfrak b)$ as graded submodules of $\mathrm{gr}_{\mathfrak m}(R)$, then $\mathfrak a=\mathfrak b$.\end{lemma}
\begin{proof} 

Since $\mathfrak a\subseteq \mathfrak b$, it suffices to show $\mathfrak b\subseteq\mathfrak a$.  We first assume that $R$ is complete.
Take $x\in\mathfrak b$ and let $t=\mathrm{ord}_{\mathfrak m}(x)$. Since the degree–$t$ pieces of the graded ideals agree, $\mathrm{in}(x)\in \big[\mathrm{gr}_{\mathfrak m}(\mathfrak a)\big]_t$. Hence there exists $a_0\in \mathfrak a\cap \mathfrak m^t$ with 
$$
[x]=a_0+y_1,\qquad y_1\in \mathfrak m^{t+1}\cap \mathfrak b.
$$
Since $y_1\in \m^{t+1},$ we consider $[y_1]\in \big[\mathrm{gr}_{\mathfrak m}(\mathfrak b)\big]_t=\big[\mathrm{gr}_{\mathfrak m}(\mathfrak a)\big]_t\subseteq  \m^{t+1}/\m^{t+2}$. Then, there exists
 $a_1\in \mathfrak a\cap \mathfrak m^{t+1}$ with
$$
[y_1]=a_1+y_2,\qquad y_2\in \mathfrak m^{t+2}\cap \mathfrak b.
$$
Therefore, $[x]=a_0+a_1+y_2$. Iterating gives
$$
x=\sum_{i=0}^{n} a_i + y_{n+1},\qquad a_i\in \mathfrak a\cap \mathfrak m^{t+i},\; y_{n+1}\in \mathfrak m^{t+n+1}\cap \mathfrak{b}.
$$

Since $R$ is complete, $\sum_{i\ge 0} a_i$ converges $\mathfrak m$-adically to some $a\in\mathfrak a$ and the remainders tend to zero. Therefore, $x=a\in\mathfrak a$. Thus $\mathfrak b\subseteq\mathfrak a$.
 
Now suppose $R$ is not complete. Then $\widehat{R}$ is faithfully flat over $R$, and we have $\mathfrak a \widehat{R} \cap R = \mathfrak a,$ and $\mathfrak b \widehat{R} \cap R = \mathfrak b.$
Moreover, $\mathrm{gr}_{\mathfrak m}(\mathfrak a)\otimes_R \widehat{R} 
\;\cong\; \mathrm{gr}_{\mathfrak m\widehat{R}}(\mathfrak a \widehat{R}),$
and similarly for $\mathfrak b$. By hypothesis 
$\mathrm{gr}_{\mathfrak m}(\mathfrak a)=\mathrm{gr}_{\mathfrak m}(\mathfrak b)$, 
hence after tensoring with $\widehat{R}$ we obtain
\[
\mathrm{gr}_{\mathfrak m\widehat{R}}(\mathfrak a \widehat{R})
= \mathrm{gr}_{\mathfrak m\widehat{R}}(\mathfrak b \widehat{R}).
\]
Complete case implies $\mathfrak a \widehat{R}=\mathfrak b \widehat{R}$. 
Intersecting with $R$ and using faithful flatness yields $\mathfrak a=\mathfrak b$.

\end{proof}

We recall that $F$-rationality behaves well under flat extensions. This was first established by V\'elez \cite{RationalVelez}, extending earlier work of Hochster and Huneke \cite{HH94}. More recently, Datta and Murayama \cite{Rationalconverse} proved that $F$-rationality descends under arbitrary flat maps of Noetherian rings, without assuming the existence of test elements. For convenience, we record below the two special cases that we need.
\begin{lemma}[{\cite[Theorem 3.1]{RationalVelez}}]\label{velezrational}
Let $\phi: R\to S $  be a smooth homomorphism of locally excellent rings of characteristic $p>0$. If $R$ is $F$-rational, then $S$ is also $F$-rational.    
\end{lemma}
\begin{lemma}[{\cite[(6) on p.440]{RationalVelez}},{\cite[Proposition A.5]{Rationalconverse}}]\label{datta}
Let $\phi: R \to S$ be a faithfully flat homomorphism of Noetherian rings of prime characteristic $p > 0$. If $S$ is $F$-rational, then so is $R$.   
\end{lemma}


As preparation for the main theorem, we recall a significant result that expresses $F$-rationality via the notions of $F$-thresholds and tight closure.

\begin{theorem}[{\cite[Corollary 3.2]{HMTW}}] \label{7}
Let $(R,\mathfrak{m},k)$ be a $d$-dimensional excellent analytically irreducible Noetherian local domain of characteristic $p > 0$, and let $J = (x_1,\ldots, x_d)$ be an ideal generated by a full system of parameters in $R$. Given an ideal $I \supseteq J$, we have $I \subseteq J^*$ if and only if $c^I(J ) = d$. In particular, $R$ is $F$-rational if and only if $c^I(J )<d$ for every ideal $I \supsetneq J$.

\end{theorem}

Having assembled the required ingredients, we are prepared to present the main theorem of this section.

\begin{corollary}\label{21}
Let $ (R, \mathfrak{m}, k)$ be a $d $-dimensional $F$-finite Noetherian local domain of characteristic $p > 0$. Let  $A = \mathrm{gr}_{\mathfrak{m}}(R)$ be the associated graded ring with maximal homogeneous ideal $\mathfrak{n}$. If $A$ is $F$-rational, then $R$ is $F$-rational.

\end{corollary}

\begin{proof}
We first reduce to the case where the residue field $k = R/\mathfrak{m}$ is infinite. If $k$ is already infinite, we proceed directly. Otherwise, we replace $R$ by the local ring 
$$
R' := R[X]_{\mathfrak mR[X]}, \qquad k' \cong k(X).
$$
Then $R'$ is a local ring with infinite residue field $k'$, and the map $R\to R'$ is smooth, hence faithfully flat. Moreover,
\[
A' := \mathrm{gr}_{\mathfrak{m}R'}(R') \;\cong\; \mathrm{gr}_{\mathfrak{m}}(R) \otimes_k k' \;\cong\; A \otimes_k k'.
\]
Since $A$ is $F$-rational and locally excellent, \autoref{velezrational} applied to the smooth extension $A \to A'$ shows that $A'$ is $F$-rational as well. Once the theorem is established for local rings with infinite residue field, it follows that $R'$ is $F$-rational. Finally, since $R \to R'$ is faithfully flat, \autoref{datta} implies that $R$ itself is $F$-rational. From now on, we assume that $k$ is infinite.

Let $J=(\mathrm{in}(x_1),\ldots, \mathrm{in}(x_d))$ be an ideal of $A$ generated by a full system of parameters such that $\bar{J}=\mathfrak{n}$, where $\bar{J}$ denotes the integral closure of $J$. Assume $\mathfrak{a}=(x_1,\ldots,x_d)$ and $\mathfrak{a} \subsetneq \mathfrak{b}$ be ideals of $R$. Using \autoref{PropGenParameterIdeal}, $\mathrm{gr}_\m(\a)=J$.  Then by \autoref{6}, we have $\bar{\mathfrak{a}}=\mathfrak{m}$. Now, consider $I = \mathrm{gr}_{\mathfrak{m}}(\mathfrak{b}) \subseteq A$. Since $\mathfrak{a} \subsetneq \mathfrak{b}$, by \autoref{22}, $J = \mathrm{gr}_{\mathfrak{m}}(\mathfrak{a})\subsetneq I=\mathrm{gr}_{\mathfrak{m}}(\mathfrak{b})$. 

Since $ \overline{\mathfrak{a}} = \mathfrak{m} $, by \autoref{9}, $c^{\mathfrak{b}}(\mathfrak{a}) = c^{\mathfrak{b}}(\mathfrak{m}).$
Similarly, since $ \bar{J} = \mathfrak{n} $, we have $c^I(J) = c^I(\mathfrak{n})$. Since $A$ is $F$-rational, $ c^I(J)<d$ by \autoref{7}. Then, by \autoref{20}, we have
\[
c^{\mathfrak{b}}(\mathfrak{a})=c^{\mathfrak{b}}(\mathfrak{m}) \leq c^I(\mathfrak{n})=c^I(J)<d.
\]
Therefore, $R$ is $F$-rational.
\end{proof}

\begin{remark}\label{otherproof}
We now present a sketch of proof  for \autoref{21} given by an anonymous referee, which we could not find in the published literature.
If $\gr_{\m}(R)$ is $F$-rational, then the extended Rees algebra $\mathcal{R} := R[\m t, t^{-1}]$ localized at the maximal ideal $(\m + \m t + t^{-1})$ is $F$-rational by the deformation property. Localizing further at the prime ideal $(\m + \m t)$ shows that $R(t)$ is $F$-rational. In particular, since $R$ is obtained from $R(t)$ by specializing $t=1$, we conclude that $R$ itself is $F$-rational.
We also note that this result was recorded by Ajit and Simper in a recent preprint \cite[Theorem 6.1(2) and Remark 6.2]{AS}.
\end{remark}

\begin{example}
Let $k$ be a field of characteristic $p>0$. Let $X=(x_{ij})$ be a $2\times3$ matrix of indeterminates and let $\Delta_{ij}$ denote the $2\times2$ minor of the columns $i<j$. Set
\[
A=\frac{k[x_{ij}]}{(\Delta_{12},\Delta_{13},\Delta_{23})}
\qquad\text{and}\qquad
R=\frac{k\llbracket x_{ij}\rrbracket}{(\Delta_{12}+M,\ \Delta_{13},\ \Delta_{23})},
\]
where $M=\prod_{i=1}^{2}\prod_{j=1}^{3}x_{ij}$ and $\mathfrak m=(x_{ij})R$.

Since $\ord_{\mathfrak m}(\Delta_{ij})=2$ for all $i<j$ while $\ord_{\mathfrak m}(M)=6$, 
\[
\mathrm{gr}_{\mathfrak m}(R)\cong \frac{k[x_{ij}]}{(\Delta_{12},\Delta_{13},\Delta_{23})}=A,
\]
with homogeneous maximal ideal $\mathfrak n=(\overline{x}_{ij})$.
Since $\ord_{\mathfrak m}(\Delta_{ij})=2$ for all $i<j$ while $\ord_{\mathfrak m}(M)=6$, 
\[
\mathrm{gr}_{\mathfrak m}(R)\cong \frac{k[x_{ij}]}{(\Delta_{12},\Delta_{13},\Delta_{23})}=A,
\]
with homogeneous maximal ideal $\mathfrak n=(\overline{x}_{ij})$.

It is well-known that $A$ is strongly $F$-regular \cite{HH87} (hence, $F$-rational), but it is not Gorenstein. Since $\mathrm{gr}_{\mathfrak m}(R)\cong A$ is $F$-rational, it follows that $R$ is $F$-rational.
\end{example}

\bibliographystyle{skalpha}
\bibliography{References1}
\end{document}